%% file: wandering6.tex
\documentclass[11pt,a4paper]{amsart}

\usepackage{a4wide}
\usepackage{amsfonts,amsthm,amsmath,amssymb}
\usepackage{amscd,color}
\usepackage{psfrag,graphicx,subfigure}
\usepackage{import} %% Figures with labels
\usepackage[shortlabels]{enumitem}

\theoremstyle{plain}
\newtheorem{lem}{Lemma}[section]
\newtheorem{prop}[lem]{Proposition}
\newtheorem{thm}[lem]{Theorem}

\newtheorem*{ThmA}{Theorem A}
\newtheorem*{ThmB}{Theorem B}
\newtheorem*{CorC}{Corollary C}
\newtheorem*{ThmD}{Theorem D}
\newtheorem*{CorE}{Corollary E}

\newtheorem*{Koebe-dist}{Koebe's Distortion Theorem}
\newtheorem*{Koebe-1/4}{Koebe's One-quarter Theorem}

\theoremstyle{definition}
\newtheorem{defn}[lem]{Definition}
\newtheorem*{defn*}{Definition}

\newtheorem*{ex*}{Example}

\newtheorem{rem}[lem]{Remark}
\newtheorem*{rem*}{Remark}
\newtheorem*{ack}{Acknowledgement}
\theoremstyle{remark}
\DeclareMathOperator{\diam}{diam}
\DeclareMathOperator{\dist}{dist}
\DeclareMathOperator{\Arg}{Arg}

\newcommand{\C}{\mathbb C}
\newcommand{\D}{\mathbb D}
\newcommand{\clC}{\widehat{\C}}
\newcommand{\chat}{\widehat{\C}}

\newcommand{\R}{\mathbb R}
\newcommand{\Z}{\mathbb Z}

\newcommand{\B}{\mathcal B}

\newcommand{\bd}{\partial}
\renewcommand{\Re}{\textup{Re}}
\renewcommand{\Im}{\textup{Im}}

\newcommand{\cals}{\mathcal{S}}
\newcommand{\calb}{\mathcal{B}}

\newcommand{\re}{\textup{Re}}

\begin{document}

\title{Fatou components and singularities of meromorphic functions}
\date{\today}

\author{Krzysztof Bara\'nski}
\address{Institute of Mathematics, University of Warsaw,
ul.~Banacha~2, 02-097 Warszawa, Poland}
\email{baranski@mimuw.edu.pl}

\author{N\'uria Fagella}
\address{Departament de Matem\`atiques i Inform\`atica, Institut de Matem\`atiques de la 
Universitat de Barcelona (IMUB) and Barcelona Graduate School of Mathematics (BGSMath). 
 Gran Via 585, 08007 Barcelona, Catalonia, Spain}
\email{nfagella@ub.edu}

\author{Xavier Jarque}
\address{Departament de Matem\`atiques i Inform\`atica, Institut de Matem\`atiques de la 
Universitat de Barcelona (IMUB), and Barcelona Graduate School of Mathematics (BGSMath).
 Gran Via 585, 08007 Barcelona, Catalonia, Spain}
\email{xavier.jarque@ub.edu}

\author{Bogus{\l}awa Karpi\'nska}
\address{Faculty of Mathematics and Information Science, Warsaw
University of Technology, ul.~Ko\-szy\-ko\-wa~75, 00-662 Warszawa, Poland}
\email{bkarpin@mini.pw.edu.pl}

\thanks{The first and fourth authors were partially supported by the Polish NCN grant decision DEC-2014/13/B/ST1/04551. The 
second and third authors were partially supported by the Spanish 
grant MTM2014-52209-C2-2-P. The fourth author was partially supported by PW grant 504/02465/1120.}

\bibliographystyle{amsalpha}

\begin{abstract} We prove several results concerning the relative position of points in the postsingular set $P(f)$ of a meromorphic map $f$ and the boundary of a Baker domain or the successive iterates of a wandering component. For Baker domains we answer a question of Mihaljevi\'c-Brandt and Rempe-Gillen. For wandering domains we show that if the iterates $U_n$ of such a domain have uniformly bounded diameter, then there exists a sequence of postsingular values $p_n$ such that $\dist(p_n,\partial U_n)\to 0$ as $n\to \infty$. We also prove that if $U_n \cap P(f)=\emptyset$ and  the postsingular set of $f$ lies at a positive distance from the Julia set (in $\C$) then any sequence of iterates of wandering domains must contain arbitrarily large disks. This allows to exclude the existence of wandering domains for some meromorphic maps with infinitely many poles and unbounded set of singular values.
\end{abstract}

\maketitle

\section{Introduction and statement of the results} \label{sec:intro}

We consider dynamical systems defined by the iteration of a meromorphic function 
\[
f: \C \to \clC,
\]
on the complex plane, especially those with an essentially singularity at infinity (transcendental).  Motivating examples of such maps are given by root-finding algorithms like, for instance, Newton's method applied to any entire transcendental map.  

In this setting the Riemann sphere $\chat$ splits into two invariant sets: the {\em Fatou set} $F(f)$, which consists of points for which the family of iterates $(f^n)$ is well defined for all $n>0$ and normal in some neighborhood; and its complement  $J(f)$ known as the {\em Julia set}. 

The connected components of the Fatou set or {\em Fatou components} which are periodic can be completely classified based on the possible limit functions of the family of iterates. Indeed, if $U$ is a periodic Fatou component, then either $U$ is a component of a basin of attraction of an attracting or parabolic cycle; or $U$ is a topological disk ({\em Siegel disk}) or annulus ({\em Herman ring}) on which the map is conjugate to an irrational rigid rotation; or the iterates of $f$ (or multiples of them) tend uniformly on compact subsets to the essential singularity at infinity, in which case $U$ is a component of a cycle of {\em Baker domains}. If a Fatou component is neither periodic nor preperiodic, then it is called a {\em wandering domain}. Neither Baker domains nor wandering components are present in the dynamical plane when $f$ is a rational map \cite{sullivan}. For classical background on the dynamics of meromorphic functions we refer to the survey \cite{bergweiler} or the articles \cite{bkl1,bkl2,bkl3,bkl4}.

The set $S(f)$ of finite {\em singular values} of $f$ plays an important role in determining the dynamics of the map. Recall that singular values are either {\em critical} or {\em asymptotic values}.  Among well known classes of transcendental maps are those with a finite set of singular values (the Speisser class $\cals$) or with a bounded set of singular values (the Eremenko--Lyubich class $\calb$). The orbits of the singular values and their accumulation points form the {\em postsingular set}
\[
P(f) = \C \cap \bigcup_{s\in \overline{S(f)}} \bigcup_{n=0}^\infty f^n(s), 
\]
where we neglect the terms which are not defined or are infinite.
The importance of singular values lies on the fact that periodic Fatou components are in some way associated to the postsingular set. Indeed, any basin of attraction of an attracting or parabolic cycle must contain a singular value, while the boundary components of Siegel disks and Herman rings are contained in the closure $\overline{P(f)}$ of the postsingular set. 

The relation between singular values and Fatou components that are specific for transcendental maps, namely Baker and wandering domains, is less clear. This problem is nowadays one of the challenges in transcendental dynamics, and the object of this paper. It is worth noticing that infinitely many singular values are necessary for these types of components to exist at all \cite{bak84,gold-keen,erem-lyub,bkl4}, and that entire maps in class $\calb$ do not have Baker domains, nor wandering domains tending to infinity under iteration \cite{bkl3,erem-lyub,rs99}. Hence, results in this direction must go beyond these classes of functions. It is known \cite{bhkm,Zheng2003} that for a wandering domain $U$, all finite limit functions of $\{f^n|_U\}$ lie in the derived set of $P(f)$. Moreover, if for a Baker or wandering domain $U$ we have $f^{np}|_U \to \infty$ as $n \to \infty$ for some $p \ge 1$, then $\infty$ is in the derived set of $\bigcup_{k=0}^{p-1} f^k(S(f))$ \cite{bak02,Zheng2003}. These results can be used to rule out the existence of wandering domains in certain situations. The existence of such domains was also excluded for some entire maps studied in \cite{stallard-no-wandering} and some meromorphic maps with finitely many poles, which are Newton methods of entire functions \cite{bergweiler-no-wandering}.

There are known examples of Baker domains on which $f$ is univalent (and hence contains no critical values), has finite degree larger than $1$ or infinite degree. We refer the reader to \cite{barfag,faghen,rippon,rippon-survey,absorb} for classification and examples of such components. Bergweiler in \cite{bergweiler-inv_dom} gave an example of a Baker domain of an entire map lying at a positive (Euclidean) distance from the postcritical set and therefore showing that a very mild relation between these objects is possible.

Examples of wandering domains, with the first one given by Baker in \cite{bak76}, are not so numerous. Most of them are constructed either using approximating theory \cite{elpathol}, by the lifting method \cite{faghen09} (that is by lifting a function with no zeros by the logarithm and therefore converting periodic components into wandering domains), or by quasiconformal surgery \cite{shiwand}. In the first case the method does not allow much control on the postsingular set of the map, while in the second case the relation of the postsingular set to the wandering domain is completely determined by the type of periodic Fatou component of the original map. As an example, take the function considered in \cite{faghen09}, defined as
\[
f(z)=2-\lambda - \log(2-\lambda) + 2 z -e^z
\]
with $\lambda=e^{2\pi i (1-\sqrt{5})/2}$, which is the lift of $g(w)= \frac{e^{2-\lambda}}{2-\lambda} we^{-w}$ (see Figure~\ref{siegel}). The map $f$ has infinitely many orbits of simply connected wandering domains on which $f$ is univalent, while the postsingular set is dense in the boundary of each of these wandering domains. To our knowledge, this is the only explicit example of a simply connected wandering domain on which $f$ is univalent. All other (lifting) examples contain critical points inside wandering components lifted from basins of attraction of attracting or parabolic cycles. We mention also the inspiring examples by Kisaka and Shishikura in \cite{shiwand}, where they construct wandering domains of eventual connectivity two for maps whose all singular values lie in preperiodic components (see also \cite{brs}).

\begin{figure}[hbt!]
\fboxsep=0.5pt
\begin{center}
\framebox{\includegraphics[width=0.45\textwidth]{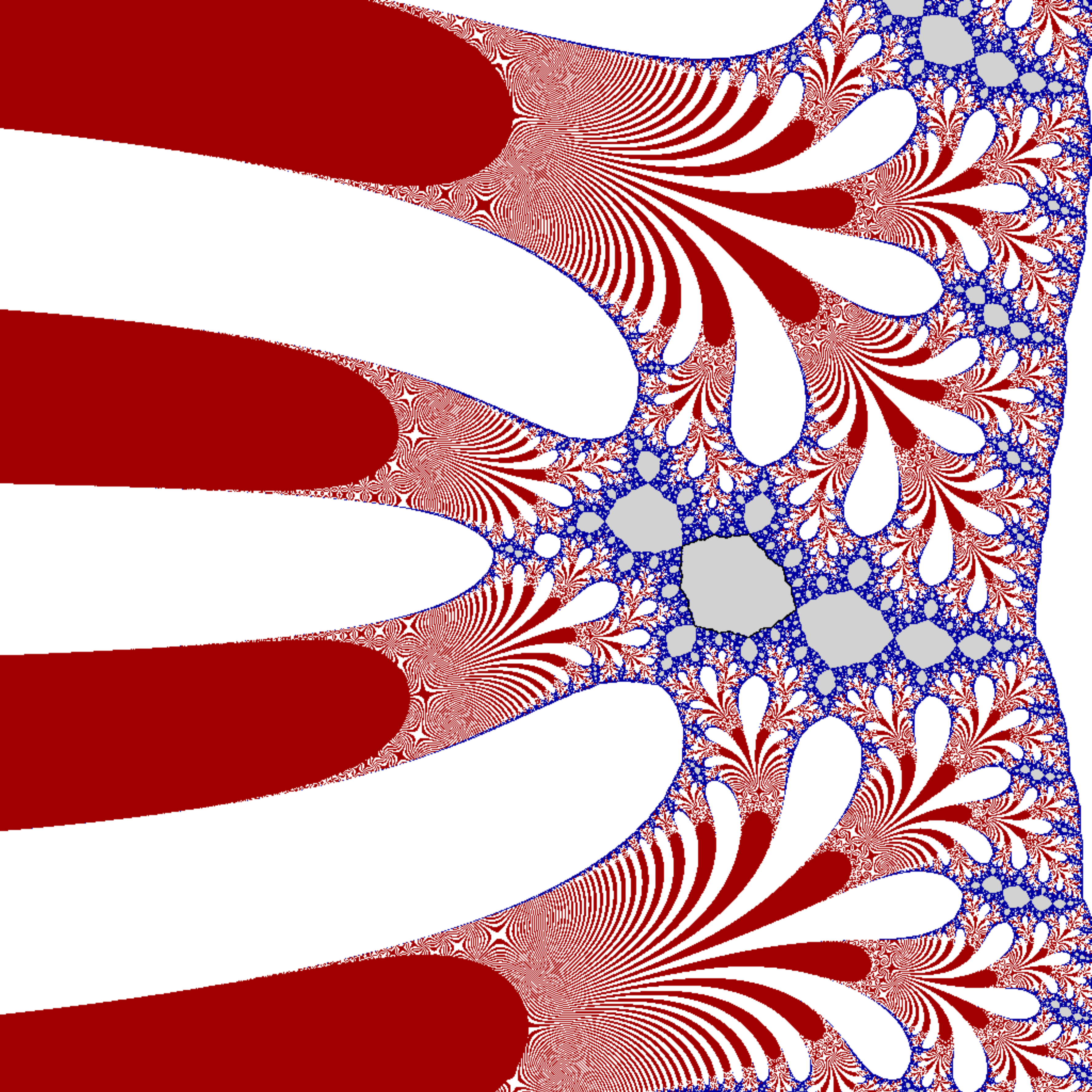}}
\framebox{\includegraphics[width=0.45\textwidth]{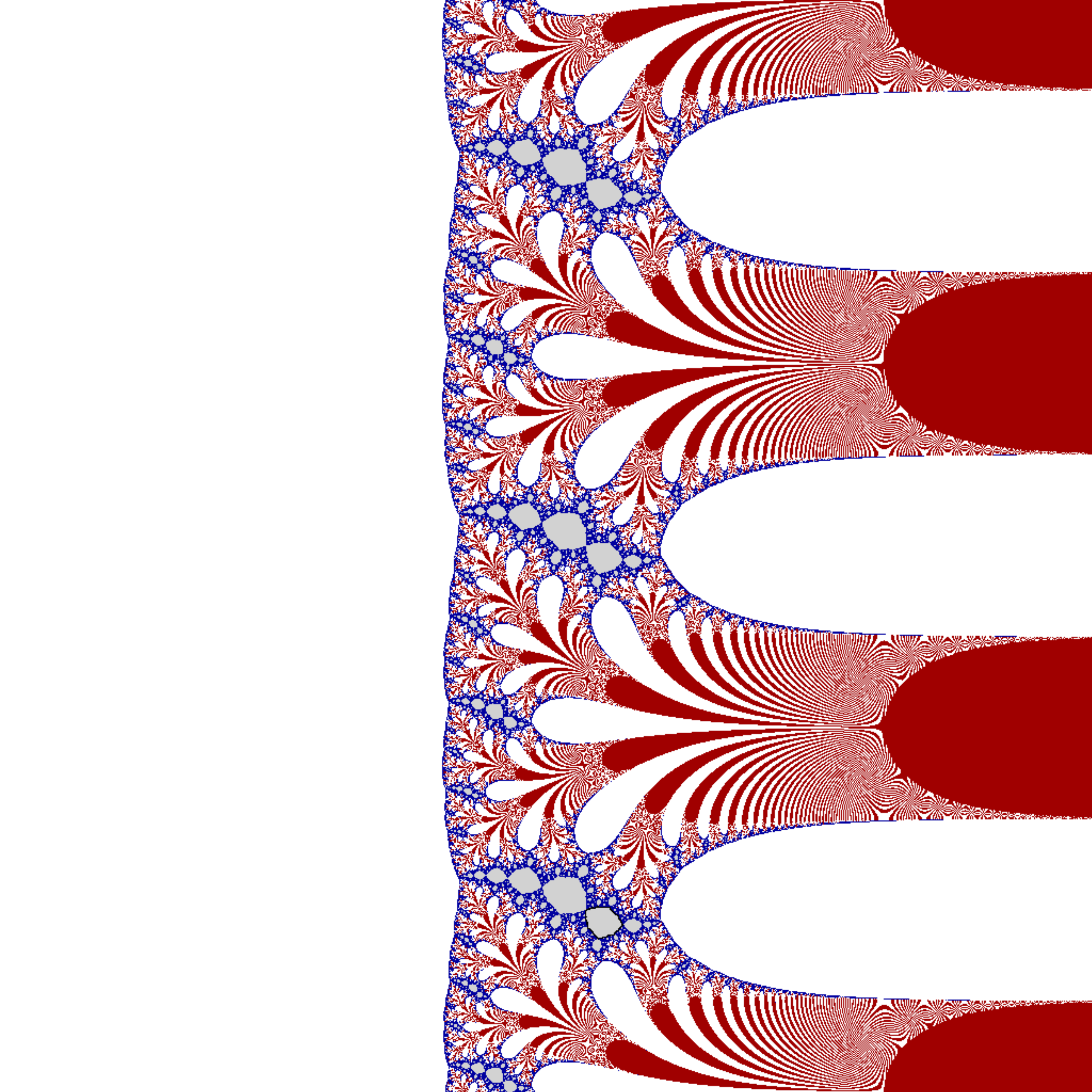}}
\end{center}
\caption{\label{siegel} \small Left: Dynamical plane of $g(w)= \frac{e^{2-\lambda}}{2-\lambda} we^{-w}$ with  $\lambda=e^{2\pi i (1-\sqrt{5})/2)}$. There is a superattracting basin centered at $w=0$ and a Siegel disk centered at $w_0=2-\lambda$. The only singular values are at the origin and at $c=2$ (on the boundary of the Siegel disk). Right: Dynamical plane of $f(z)$ satisfying $\exp(f(z))=g(\exp(z))$. The superattracting basin lifts to a Baker domain (white), while the Siegel disk lifts to infinitely many orbits of wandering domains on which $f$ is univalent (grey). Figures are extracted from \cite{faghen}.} 
\end{figure}

The above examples show different possibilities for the relationship between Baker  or wandering domains and the set $P(f)$. Our goal in this paper is to present some new results in this direction.

Our first result concerns Baker domains. In \cite{bergweiler-inv_dom}, Bergweiler showed the following result for invariant Baker domains of entire functions.

\begin{thm}[{\cite[Theorem 3]{bergweiler-inv_dom}}]\label{thm:berg}
Let $f$ be a transcendental entire function with an invariant Baker domain $U$. If $U\cap S(f) = \emptyset$, then there exists a sequence $p_n\in P(f)$ such that, as $n\to \infty$, 
\begin{enumerate}
\item[\rm (1)] $|p_n|\to\infty$, 
\item[\rm (2)] $|p_{n+1}/p_n| \to 1$,
\item[\rm (3)] $\dist(p_n,U)=o(|p_n|)$.
\end{enumerate}
\end{thm}
Here and in the sequel $\dist$ denotes the Euclidean distance. Additionally, as we mentioned above, Bergweiler provided an example (which inspired the example in Figure~\ref{siegel}), where the postcritical set is located at a positive distance from the Baker domain, showing in some sense the sharpness of the theorem. We remark that Theorem~\ref{thm:berg} holds for entire functions and therefore all considered Baker domains are simply connected. A version of Theorem~\ref{thm:berg} for meromorphic transcendental maps (and hence for Baker domains which are not necessarily simply connected) was proven by Mihaljevi\'c-Brandt and Rempe-Gillen in \cite[Theorem~1.5]{lasse-helena}, with conclusion~(2) replaced by
\[
\sup_{n\geq 1} \left|\frac{p_{n+1}}{p_n}\right| < \infty.
\]
They also asked a question, whether the stronger conclusion~(3) holds in this setting. In this paper we answer this question positively under a mild additional hypothesis, proving the following.

\begin{ThmA}
Let $U$ be a Baker domain of period $p \ge 1$ for a meromorphic map $f$, such that $f^{pn}|_U \to \infty$ as $n \to \infty$ and $\C \setminus U$ has an unbounded connected component. Then there exists a sequence of points $p_n \in P(f)$ such that 
\[
|p_n| \to \infty, \qquad \left|\frac{p_{n+1}}{p_n}\right| \to 1, \qquad \frac{\dist(p_n, U)}{|p_n|} \to 0
\]
as $n \to \infty$.
\end{ThmA}

We remark that the assumption on an unbounded component of $\C\setminus U$ (which is equivalent to the condition that $\{\infty\}$ is not a singleton component of $\clC \setminus U$) is trivially satisfied if the Baker domain is simply connected. This is always the case if $f$ is  Newton's method of an entire map, as it was shown by the authors in \cite{bfjk}.

Theorem~A is a consequence of  a technical lemma (Lemma \ref{lem:1})  which generalizes  \cite[Lemma 3]{bergweiler-inv_dom} and, additionally, has several applications concerning the relation between points in $P(f)$ and wandering domains. The following consequence of Lemma \ref{lem:1}, although stated for a general Fatou component, is most meaningful when $U$ is a wandering or Baker domain. In particular, it contributes to an answer to Question~11 from \cite{bergweiler}.

\begin{ThmB} \label{cor:interval}
Let $f$ be a transcendental meromorphic map and $U$ be a Fatou component of $f$. Denote by  $U_n$ the Fatou component such that $f^n(U) \subset U_n$.  Then  for  every $z\in U$ there exists a sequence $p_n \in P(f)$ such that
\[
\frac{\dist(p_n, U_n)}{\dist(f^n(z), \partial U_n)} \to 0, \quad \text{ as $n\to \infty$.} 
\]
In particular, if for some $d > 0$ we have $\dist(f^n(z),\partial U_n) < d$ for all $n$ $($for instance if the diameter of $U_n$ is uniformly bounded$)$, then $\dist(p_n, U_n)\to 0$ as $n$ tends to $\infty$. 
\end{ThmB}

The second application of Theorem~B concerns functions for which the Julia set and the postsingular set are apart from each other. More precisely, following Mayer and Urba\'nski \cite{mu07,mu10} we consider the class of {\em topologically hyperbolic} meromorphic functions. 
This class was also considered by Stallard in \cite{stallard-classC} in the context of entire maps, where she proved that in this case $|(f^n)'(z)| \to \infty$ for all $z \in J(f)$.

\begin{defn}
A meromorphic transcendental function $f$ is called {\em topologically hyperbolic} if
\[ 
\dist(P(f), J(f)\cap \C) >0.
\] 
\end{defn}

This condition can be regarded as a kind of weak hyperbolicity in the context of transcendental meromorphic functions, while hyperbolicity is (usually) defined by the condition that $\overline{P(f)}$ is bounded and disjoint from the Julia set.  In particular, hyperbolic maps are always in class $\B$. On the contrary, topologically hyperbolic maps could have sequences of (post)singular points in the Fatou set accumulating at infinity. See e.g.~the discussion on various kinds of hyperbolicity for transcendental maps in \cite{rs99}. Note also that topologically hyperbolic maps do not possess parabolic cycles, rotation domains or wandering domains which do not tend to infinity (otherwise they would have a constant finite limit function in $\overline{P(f)}\cap J(f)$, see \cite{bak02}). Examples of topologically hyperbolic maps include many Newton's methods of entire functions (see Propositions~\ref{prop:sine_tangent} and~\ref{prop:the_other}). 

The following corollary shows that Baker or wandering domains of topologically hyperbolic maps either contain postsingular points or arbitrarily large disks.  

\begin{CorC}\label{cor:disk}
Let $f$ be a topologically hyperbolic meromorphic map and $U$ be a Fatou component of $f$.  Denote by $U_n$ the Fatou component such that $f^n(U) \subset U_n$ and suppose that $U_n\cap P(f)=\emptyset$ for $n > 0$. Then for every compact set $K\subset U$, every $z\in K$ and every $r>0$ there exists $n_0$ such that for all $n\geq n_0$ 
\[
\mathbb D(f^n(z),r) \subset U_n.
\]
In particular, 
\[
\diam U_n \to \infty \quad \quad \text{and} \quad \quad \dist(f^n(z), \bd U_n) \to \infty
\]
for every $z\in U$, as $n \to \infty$.

\end{CorC}

Corollary~C may be used in many instances to exclude the presence of wandering domains for a given map. Examples can be found in Section~\ref{sec:ex}, where we show that certain Newton maps outside class $\calb$, with infinitely many poles, do not possess wandering Fatou components. In particular, we prove (see Theorem~D and Corollary~E) that Newton's method of the function $h(z) = ae^z + b z + c$ has no wandering domains for many parameters $a, b, c$, including the case when $a,b,c \in \R$, $a,b \neq 0$ and the map $h$ has a real zero.

\begin{ack}
The authors would like to thank the Centro Internazionale per la Ricerca Matematica (CIRM) in Trento for the hospitality during their visit within Research in Pairs program.
\end{ack}

%%%%%%%%%%%%%%%%%%%%%
%%%%%%%%%%%%%%%%%%%%%

\section{Preliminaries and tools} \label{sec:prelim}

We use the notation $\D(z, r) = \{w \in \C: |w-z| < r\}$ and 
\[
\D(A,r) = \bigcup_{z \in A} \D(z, r)
\]
for $z \in \C$, $A \subset \C$, $r >0$. By $[z_1, z_2]$ we denote the straight line segment connecting $z_1$ to $z_2$. The length of a rectifiable curve $\gamma$ is denoted by $\ell(\gamma)$. The symbol $\diam$ denotes the Euclidean diameter. We set
\[
\dist(z, B) = \inf\{|z-w|:  w \in B\}, \quad \dist(A, B) = \inf\{|z-w|: z\in A, \: w \in B\}
\]
for $z \in \C$, $A, B \subset \C$. The closure and boundary of a set $A \subset \C$ are denoted, respectively, by $\overline{A}$ and $\bd A$.

We use the following classical results on the distortion of conformal maps (see e.g.~\cite{carlesongamelin}). 

\begin{Koebe-dist} \label{koebe-dist}
Let $g:\D(a,r)\to\C$ be a univalent holomorphic map, $\rho \in (0,1)$ and
$z\in \overline{\D(a,\rho r)}$. Then
\begin{equation*}
\frac{1-\rho}{(1+\rho)^3}
|g'(a)|
\leq |g'(z)|
\leq
\frac{1+\rho}{(1-\rho)^3}
|g'(a)|.
\end{equation*}
\end{Koebe-dist}

\begin{Koebe-1/4}
Let $g:\D(a,r)\to\C$ be a univalent holomorphic map. Then
\begin{equation*}\label{koebe-1/4}
g(\D(a,r))\supset
\D\left(g(a),\tfrac14 |g'(a)|r\right).
\end{equation*}
\end{Koebe-1/4}

In this section we prove the following lemma, which is the main tool in the proofs of our results. It is a generalization of \cite[Lemma 3]{bergweiler-inv_dom}. The estimates are similar to the ones in \cite[Proposition 7.4]{lasse-helena}, which were formulated in terms of the hyperbolic metric.

\begin{lem}\label{lem:1}
Let $f$ be a transcendental meromorphic map and $U$ be a Fatou component of $f$.  Denote by $U_n$ the Fatou component such that $f^n(U) \subset U_n$. Then for every compact set $K \subset U$ and every $\varepsilon > 0$, $M \ge  1$ there exists $n_0 > 0$ such that for every $n \ge n_0$, every $z \in K$ and every curve $\gamma$ in $\C$ connecting $f^n(z)$ to a point $w \in \bd U_n$ with $\ell(\gamma) \le M \dist(f^n(z), \bd U_n)$ there exists a point
\[
p \in \D(\gamma, \varepsilon \ell(\gamma)) \cap P(f).
\]
\end{lem}

\begin{proof} We argue by contradiction. If the assertion of the lemma does not hold, then there exist a compact set $K \subset U$, numbers $\varepsilon > 0$, $M \ge 1$, sequences $n_j \to \infty$, $z_j \in K$ and curves $\gamma_j$ connecting the point $\xi_j := f^{n_j}(z_j)$ to $w_j \in \bd U_{n_j}$, such that 
\begin{equation}\label{eq:ass1}
\ell(\gamma_j) \le M \dist(\xi_j, \bd U_{n_j})
\end{equation}
and $\D(\gamma_j, \varepsilon \ell(\gamma_j))\cap P(f) = \emptyset$. Replacing $\varepsilon$ by $2\varepsilon$, we can actually assume
\begin{equation}\label{eq:ass2} 
\D(\gamma_j, \varepsilon \ell(\gamma_j))\cap \overline{P(f)} = \emptyset.
\end{equation}
Take the branch $g_j$ of $f^{-n_j}$ defined on $\D(\xi_j, \varepsilon \ell(\gamma_j))$, such that $g_j(\xi_j) = z_j$ and extend it along $\gamma_j$. Then $g_j$ is well-defined (single-valued) on 
\[
D_j := \D(\xi_j, \varepsilon \ell(\gamma_j)) \cup \gamma_j \cup \D(w_j, \varepsilon \ell(\gamma_j)).
\]
We claim that the distortion of $g_j$ on $\gamma_j$ (i.e.~$\sup_{u_1, u_2 \in \gamma_j} |g_j'(u_1)|/|g_j'(u_2)|$) is bounded by a constant independent of $j$.
To see the claim, take a parameterization of $\gamma_j$ by $t\in [0,1]$ and define the sequence $t_0, \ldots, t_k \in [0,1]$ inductively in the following way. Let $t_0 = 0$ and let 
\[
t_{s+1} = \inf\left\{t \in (t_s, 1]: |\gamma_j(t) - \gamma_j(t_s)| = \frac{\varepsilon \ell(\gamma_j)}{2}\right\}
\]
for $s\ge 0$, as long as the infimum is defined (i.e. the above set is non-empty). We notice that if $t_s$ is defined, then
\[
s\frac{\varepsilon\ell(\gamma_j)}{2} = 
|\gamma_j(t_1) - \gamma_j(t_0)| + \cdots +  |\gamma_j(t_s) - \gamma_j(t_{s-1})| \leq \ell(\gamma_j),
\]
which implies $s \le 2/\varepsilon$. This means that there exists $k >0$ such that $k \le 2/\varepsilon$ and $t_{k+1}$ is not defined, i.e.~$|\gamma_j(t) - \gamma_j(t_k)| < \varepsilon \ell(\gamma_j)/2$ for every $t \in (t_k, 1]$. Setting $t_{k+1} = 1$, we have a sequence $0 = t_0 <  \cdots < t_{k+1} = 1$, such that
\[
\gamma_j([t_s, t_{s+1}]) \subset \overline{\D\left(\gamma_j(t_s), \frac{\varepsilon \ell(\gamma_j)}{2}\right)}
\]
for every $s = 0, \ldots, k$. Since $g_j$ extends to $\D(\gamma_j(t_s), \varepsilon \ell(\gamma_j))$, the Koebe Distortion Theorem implies that the distortion of $g_j$ on $\gamma_j([t_s, t_{s+1}])$ is bounded by a constant $c_0>1$, which is independent of $j$ and $s$. Then the distortion of $g_j$ on $\gamma_j$ is bounded by $c_1 :=c_0^{k+1} \le c_0^{1 + 2/\varepsilon}$.

Let
\[
v_j = g_j(w_j), \qquad r_j = \sup\{r > 0: \D(v_j, r) \subset g_j(D_j)\}.
\]
Since $\D(w_j, \varepsilon \ell(\gamma_j)) \subset D_j$, by the Koebe One-quarter Theorem, we have
\begin{equation}\label{eq:r_j}
r_j \ge \frac{\varepsilon\ell(\gamma_j)}{4} |g_j'(w_j)|.
\end{equation}
Moreover,
\begin{equation}\label{eq:c}
|z_j - v_j| \le { \sup_{u\in\gamma_j}|g_j'(u)|} \; \ell(\gamma_j)  \le c_1 |g_j'(w_j)|\; \ell(\gamma_j).
\end{equation}
As $v_j \notin U$, $z_j \in K$ and $K \subset U$ is compact, we have
\[
|z_j - v_j| > c_2
\]
for some $c_2 > 0$ independent of $j$, which together with \eqref{eq:r_j} and \eqref{eq:c} gives
\begin{equation}\label{eq:r}
r_j > c_3{:=\frac{\varepsilon c_2}{4c_1}}
\end{equation}
independent of $j$. By \eqref{eq:ass1} and \eqref{eq:ass2}, we have 
\[
\D(\xi_j, \min(\varepsilon, 1/M) \ell(\gamma_j)) \subset D_j \cap U_n,
\]
and so, by the Koebe One-quarter Theorem, we conclude
\begin{equation}\label{eq:inU}
\D\left(z_j, \frac{\min(\varepsilon, 1/M) \ell(\gamma_j)}{4} |g_j'(\xi_j)|\right) \subset U.
\end{equation}
On the other hand, as $K \subset U$ is compact, we have $\max_{z\in K} \dist(z,\partial U)<\infty$, so \eqref{eq:inU} implies
\begin{equation}\label{eq:c4}
\ell(\gamma_j) |g_j'(\xi_j)| < c_4
\end{equation}
for some $c_4 > 0$ independent of $j$. This, together with \eqref{eq:c} and the compactness of $K$, gives
\[
|v_j| \le |z_j - v_j| + |z_j| \le c_5
\]
for some $c_5 > 0$ independent of $j$. Hence, taking a subsequence of $n_j$ we can assume that $v_j \to v \in \bd U \cap \C$, and by \eqref{eq:r} we have
\[
D:=\D\left(v, \frac{c_3} 2\right) \subset g_j(D_j)
\]
for every sufficiently large $j$. In particular, $v \in D \cap J(f)$ and $f^{n_j}$ are defined and univalent on $D$ for all large $j$, with the distortion bounded by a constant independent of $j$. By the density of periodic sources in $J(f)$, we have $|(f^{n_j})'| \to \infty$ on $D$ as $j \to \infty$, so by the Koebe 1/4 Theorem, every bounded set in $\C$ is contained in $f^{n_j}(D)$ for sufficiently large $j$, which clearly contradicts the univalency of $f^{n_j}$ on $D$.
\end{proof}

%%%%%%%%%%%%%%%%%%%%%%%%%%%%%%
\section{Proofs of main results}\label{sec:A}

\begin{proof}[Proof of Theorem \rm A]

Fix a point $z \in U$ and let $K$ be a curve connecting $z$ to $f^p(z)$ in $U$. Then 
\[
\Gamma = \bigcup_{m = 0}^\infty f^{pm}(K) \subset U
\]
is a curve starting at $z$ and tending to $\infty$. Since $\C \setminus U$ has an unbounded component,  
there exists $r_0 > 0$ such that for every $r > r_0$ there is a point $w \in \bd U$ with $|w| = r$. Replacing $z$ by $f^{pn}(z)$ for some $n > 0$ we can assume 
\begin{equation}\label{eq:r_0}
|\zeta| > r_0 \quad \text{for every} \quad \zeta \in \Gamma.                                                                                                                                                            \end{equation}

Take a positive integer $k$ and use Lemma~\ref{lem:1} for the set $K$ and 
\[
\varepsilon = \frac 1 k, \quad M = 2\pi k, 
\]
to find $N(k)$ such that for every $m \ge N(k)$ and every curve $\gamma$  connecting $f^{pm}(z)$ to a point in $\bd U$ with $\ell(\gamma) \le M \dist(f^{pm}(z), \bd U)$ there exists a point in $\D(\gamma, \varepsilon \ell(\gamma)) \cap P(f)$.
Increasing $N(k)$ inductively, we can assume 
\[
|f^{pN(k)}(z)| < |f^{pN(k+1)}(z)|
\]
for all $k$.

\begin{figure}[hbt!]
\centering
\def\svgwidth{0.9\textwidth}
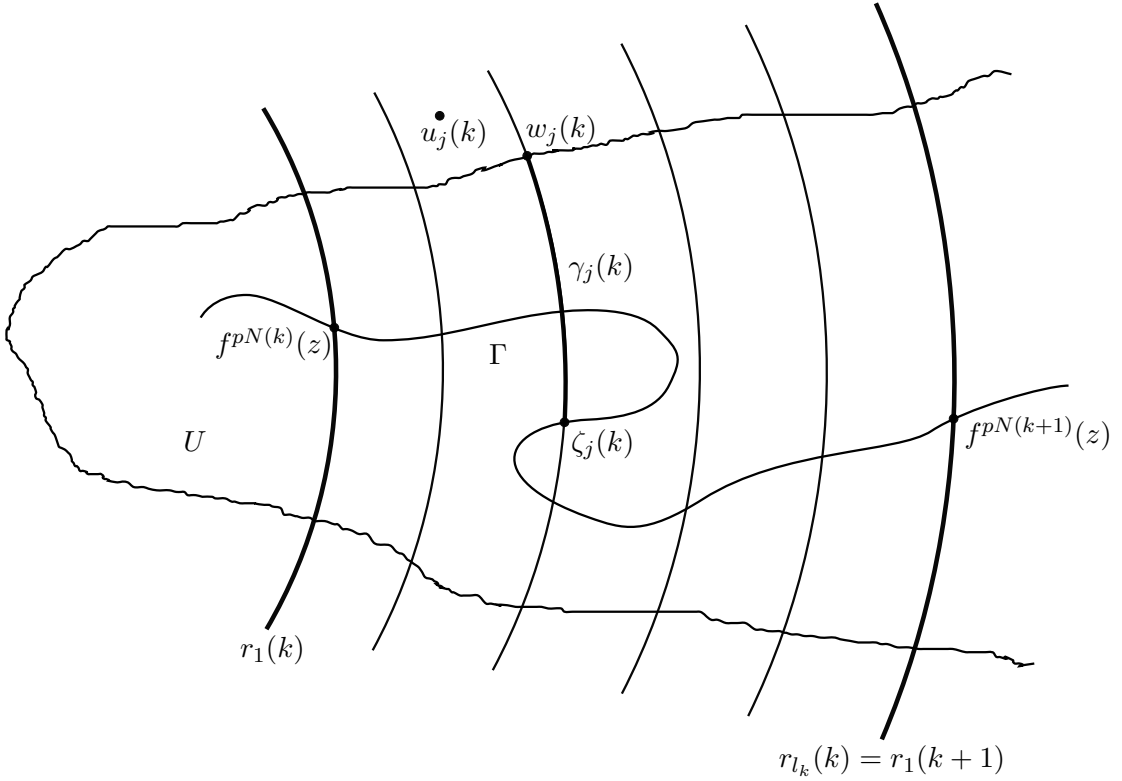
\caption{Construction of points $u_j(k)$.}
\label{fig:baker}
\end{figure}

Choose positive numbers
\[
r_1(k) < \ldots < r_{l_k}(k) \in \left[|f^{pN(k)}(z)|, |f^{pN(k+1)}(z)|\right]
\]
such that $r_1(k) = |f^{pN(k)}(z)|$, $r_{l_k}(k) = |f^{pN(k+1)}(z)|$ and
\begin{equation}\label{eq:r(j)}
\frac{r_{j+1}(k)}{r_j(k)} < 1 + \frac{1}{k}, \qquad j = 1, \ldots, l_k - 1
\end{equation}
Note that, by definition, 
\begin{equation}\label{eq:r_k}
r_{l_k}(k) = r_1(k+1)
\end{equation}
and
\begin{equation}\label{eq:to_infty}
\lim_{k \to \infty} \min_j r_j(k) = \lim_{k \to \infty} r_1(k) = \infty.
\end{equation}

Take $j \in \{1, \ldots, l_k-1\}$. Then there is a point $\zeta_j(k) \in \bigcup_{m = N(k)}^{N(k+1)} f^{pm}(K)$ with 
\[
|\zeta_j(k)| = r_j(k).
\]
Now define a curve $\gamma_j(k)$ joining $\zeta_j(k)$ to $\bd U$ in the following way. If $\dist(\zeta_j(k), \bd U) \le r_j(k)/k$, then we set $\gamma_j(k)$ to be the straight line segment from $\zeta_j(k)$ to the nearest point of $\bd U$. In this case we have 
\[
\ell(\gamma_j(k)) = \dist(\zeta_j(k), \bd U) \le \frac{r_j(k)}{k}.
\]
In the complementary case $\dist(\zeta_j(k), \bd U) > r_j(k)/k$, note that $r_j(k) > r_0$ by \eqref{eq:r_0}, so there is a point $w_j(k) \in \bd U$ with 
\[
|w_j(k)| = r_j(k).
\]
Let $\gamma_j(k)$ be a circle arc in $\bd \D(0,r_j(k))$ connecting $\zeta_j(k)$ to $w_j(k)$. 
Changing $w_j(k)$ if necessary, we can assume $\gamma_j(k) \setminus \{w_j(k)\} \subset U$.
In this case we have 
\[
\ell(\gamma_j(k)) \le 2\pi r_j(k) \le M \dist(\zeta_j(k), \bd U).
\]
Since in both cases $\zeta_j(k) \in f^{pm}(K)$ for $m \ge N(k)$ and  $\ell(\gamma_j(k)) \le M \dist(\zeta_j(k), \bd U)$, by Lemma~\ref{lem:1} and the definition of $N(k)$ we conclude that there is a point 
\[
u_j(k) \in P(f) \cap \D\left(\gamma_j(k), \frac{\ell(\gamma_j(k))}{k}\right).
\]
By the definition of $\gamma_j(k)$,
\[
\gamma_j(k) \subset \overline{\D\left(0, r_j(k) + \frac {r_j(k)} k\right)} \setminus \D\left(0, r_j(k) - \frac {r_j(k)} k\right) 
\]
and
\[
\ell(\gamma_j(k)) \le 2 \pi r_j(k),
\]
so
\begin{equation}\label{eq:u1}
1 - \frac{2\pi + 1}{k} < \frac{|u_j(k)|}{r_j(k)} < 1 + \frac{2\pi +1}{k},
\end{equation}
and by \eqref{eq:r(j)} and \eqref{eq:r_k},
\begin{equation}\label{eq:u2}
\left|\frac{u_{j+1}(k)}{u_j(k)}\right| < 1 + \frac{c}{k}
\end{equation}
for some $c > 0$ independent of $k,j$, where we define
\[
u_{l_k}(k) = u_1(k+1).
\]
Moreover, by the definition of $\gamma_j(k)$, 
\begin{equation}\label{eq:u3}
\dist(u_j(k), U) \le \frac{2 \pi r_j(k)}{k}.
\end{equation}

To define a suitable sequence of points $p_n \in P(f)$ which satisfies the assertions of the theorem, it is sufficient to renumber 
the sequence
\[
u_1(1), \ldots, u_{l_1 - 1}(1), u_1(2), \ldots, u_{l_2 - 1}(2), \ldots.
\]
This can be done formally by setting
\[
p_n = u_{j(n)}(k(n)), \ldots, n = 1, 2, \ldots,
\]
where $j(n), k(n)$ are the unique numbers satisfying
\[
1 + \sum_{s = 1}^{k(n)-1} (l_s - 1) \le n <  1 + \sum_{s = 1}^{k(n)} (l_s - 1) \quad \text{and} \quad j(n) = n - \sum_{s = 1}^{k(n)-1} (l_s - 1).
\]
Clearly, $k(n) \to \infty$ as $n \to \infty$. 
By \eqref{eq:to_infty} and \eqref{eq:u1}, $|p_n| \to \infty$ as $n \to \infty$, while \eqref{eq:u2} implies $|p_{n+1}/p_n| \to \infty$. Finally, \eqref{eq:u1} and \eqref{eq:u3} give $\dist(p_n, U)/|p_n| \to 0$. The proof of Theorem~A is complete.

\end{proof}
\begin{rem} It would be interesting to determine whether the assumption that $\C \setminus U$ has an unbounded connected component is necessary.
\end{rem}

\begin{proof}[Proof of Theorem~{\rm B}]
Take a positive integer $k$ and $z\in U$. We use Lemma~\ref{lem:1} for $\varepsilon = 1/k$, $M = 1$ and $\gamma_n = [f^n(z), w_n]$, where 
\[
|f^n(z) - w_n| = \dist(f^n(z), \bd U_n) =: r_n.
\]
By Lemma~\ref{lem:1} we know that there exists $n_0(k)$, such that for every $n \ge n_0(k)$ one can find a point 
\[
p_n \in \D\left([f^n(z), w_n], \frac{r_n}{k}\right) \cap P(f).
\]
Clearly, we can assume $n_0(k) \to \infty$ as $k \to \infty$. Since $[f^n(z), w_n] \setminus \{w_n\} \subset U_n$ by definition, we have 
\begin{equation}\label{eq:in}
\dist(p_n, U_n) <\frac{r_n}{k}.
\end{equation}
Now for any $n >0$ define $k(n) = \max\{k: n_0(k) \le n\}$. Note that the maximum is attained since $n_0(k) \to \infty$ as $k \to \infty$. Obviously, $k(n) \to \infty$ as $n \to \infty$ and $n \ge n_0(k(n))$, so by \eqref{eq:in} we have
\[
\frac{\dist(p_n, U_n)}{r_n} <  \frac{1}{k(n)} \to 0
\]
as $n \to \infty$, which ends the proof.
\end{proof}
 
\begin{proof}[Proof of Corollary~{\rm C}]
By Theorem~B, there exists a sequence of points $p_n \in P(f)$ such that
\[
\frac{\dist(p_n, U_n)}{\dist(f^n(z), \bd U_n)} < \varepsilon_n,
\]
where $\varepsilon_n \to 0$ as $n \to \infty$. Since $f$ is topologically hyperbolic and $U_n \cap P(f) = \emptyset$, we have $\dist(p_n, U_n) > c$ for some fixed $c > 0$, which implies $\dist(f^n(z), \bd U_n) > c/\varepsilon_n \to \infty$.
\end{proof}

\section{Examples}\label{sec:ex}

As an application of the above results we will consider some examples of topologically hyperbolic maps for which we can rule out the existence of wandering domains. All the examples are Newton's methods for some transcendental entire functions. Some of the maps considered in this section were studied in detail in  \cite[Examples~7.2--7.3]{bfjkaccesses}. Here we show additionally that none of them has wandering domains. 

The first proposition was proved previously by Bergweiler and Terglane in \cite{berter} by the use of Sullivan's quasiconformal deformations technique. Here we present an elementary proof based on Corollary~C.

\begin{prop} \label{prop:sine_tangent}
The map
\[
N_f(z)=z-\tan(z),
\]
which is Newton's method of the map $f(z)=\sin (z)$, has no wandering domains. 
\end{prop}

\begin{proof}

The map $N_f$ has all its fixed points located at $c_k:=k\pi,\ k\in \mathbb Z$ (the zeros of $\sin(z)$). All of them are critical points (in fact $N_f'(c_k)=N_f^{\prime\prime}(c_k)=0$), and they are the only ones. 
Moreover, $N_f$ has no finite asymptotic values. 
One can also prove that the vertical lines 
$$
r_k(t):=\frac{\pi}{2}+k\pi + {\rm i}t, \quad t\in \mathbb R, \ k\in \mathbb Z,
$$
are invariant, contain all poles of $N_f$ (located at $p_k:=\frac{\pi}{2}+k\pi$) and are contained in the Julia set of $N_f$ (see \cite{bfjkaccesses}). In each of the vertical strips bounded by the lines $r_k$ there is an (unbounded) attracting basin of the super-attracting fixed point $c_k$. Since
$$
N_f^{\prime}(z)=1-\frac{1}{\cos^2(z)},
$$
it is clear from the periodicity of the function that there exists $r>0$ such that for every $k$, we have $|N_f^{\prime}(z)|\leq 1/2$ for $z$ in $\mathbb D(c_k,r)$ and we conclude that this disk is contained in the Fatou set of $N_f$. 

It follows that $N_f$ is topologically hyperbolic and therefore if $W$ is a wandering domain for $N_f$ and $z\in W$, then $N_f^n(z)$ tends to infinity as $n\to\infty$. Moreover, Corollary~C implies that for any $R>0$ we have $\mathbb D(N_f^n(z),R)\subset W_n$ for every sufficiently large $n$, where $W_n$ is the Fatou component such that $N_f^n(W)\subset W_n$. This contradicts  the fact that the lines $r_k$, which are at distance $\pi$ from each other, are contained in the Julia set. 

\end{proof}

The second proposition is proved in a similar way.

\begin{prop} \label{prop:the_other}
The map 
\[
N_g(z)=z + i + \tan(z),   
\]
which is the Newton's method of the map $g(z)=\exp\left(-\int_{0}^{z} \frac{\text{\rm d}u}{i + \tan(u)} \right)$,  
has no wandering domains.
\end{prop}

\begin{proof}
In \cite{bfjkaccesses} it is shown that $N_g$ has an invariant Baker domain $U$ such that
$$
\mathbb H^+ \cup L  \subset U,
$$
where $\mathbb H^+$ is the upper halfplane and
$$
L = \bigcup_{k\in \mathbb Z}l_k \quad \text{for} \quad l_k=\{\Re(z)=k\pi,\ k\in \mathbb Z \}.
$$
Note also that the vertical lines
\[
s_k = \left\{z \in\C: \Re (z) =\frac \pi 2+k\pi\right\}
\]
are invariant under $N_g$. Each of the lines contains the pole $p_k:=\pi/2+k\pi$ and the two critical points $c_k=\pi/2+k\pi\pm{ i}y_k$, where $\pm y_k\approx \pm 0.8814$ are, respectively, the solutions of the equation $\exp(-y)-\exp(y)=\pm 2$. Moreover, $N_g$ has no other critical points and no finite asymptotic values. From the equality 
\begin{equation*}
\tan (x+i y)= \frac{\sin 2x}{\cos 2x + \cosh 2y} + 
i \left( \frac{\sinh 2y}{\cos 2x + \cosh 2y}\right)
\end{equation*}
it follows that all points $z$ in the halflines $s_k^{-} :=\{z\in s_k: \Im(z) < 0\}$ satisfy $\Im (N_g(z)) <  \Im (z)$. Since $N_g$ has no fixed points, this implies that $\Im(N_g^n(z))$ tends to $-\infty$ as $n \to \infty$ for $z\in s_k^-$. Moreover, one can check that the sets
\[
S_k := \left\{z\in\C : \left|\Re (z) - \frac \pi 2 - k\pi\right| < \frac \pi 8, \; \Im(z) < -\frac {\ln 2} 2 \right\} 
\]
satisfy $N_g(\overline{S_k}) \subset S_k$ (see~\cite{bfjkaccesses}), which implies that $s_k^- \cup S_k$ is contained in an invariant Baker domain of $N_g$. We conclude that $N_g$ is topologically hyperbolic. Hence, $N_g$ has no wandering domains since for every such domain $W$, its forward images $W_n:=f^n(W)$ would contain arbitrary large discs which cross the lines $l_k$, a contradiction.   
\end{proof}

The main result of this section is to exclude the presence of wandering domains for the following class of maps. 

\begin{ThmD}Let $h(z)=ae^z + bz + c$ for $a,b,c \in \C$, $a,b \neq 0$ and let 
\begin{equation}\label{eq:example_h}
N_h(z) = z - \frac{h(z)}{h'(z)} =\frac{z-1 - \alpha e^{-z}}{1+ \beta e^{-z}},
\end{equation}
where $\alpha = c/a$, $\beta = b/a$, be its Newton's method. If $a,b,c \in \C$, $a,b \neq 0$, and the asymptotic value $u = -c/b = -\alpha/\beta$ of $N_h$ satisfies $\inf_{n\ge 0}\dist(N_h^n(u), J\left(N_h\right)) > 0$, then $N_h$ has no wandering domains.
\end{ThmD}

\begin{rem}\label{rem:abc} Checking the relation $\alpha = c/a$, $\beta = b/a$ is straightforward. Observe that for $a = 0$ the map $N_h$ degenerates to a constant. If $a \neq 0$ and $b = c = 0$, then $N_h(z) = z - 1$, and if $a, c \neq 0$ and $b = 0$, then  $N_h$ is conjugated to a Fatou function of the form $F(z) = z +e^{-z} + \lambda$, for some $\lambda \in \C$. In \cite{kotus-urbanski-fatou} it is proved that such maps have no wandering domains for $\Re(\lambda) \ge 1$. Here we assume $a,b \neq 0$. Since $\alpha = c/a$, $\beta = b/a$, it is enough to consider the case $a = 1$ (and set $\alpha = c$, $\beta = b \ne 0$). Moreover, it is easy to check that $N_h$ has exactly one asymptotic value $u = -\alpha/\beta$, and it satisfies $u = \lim_{t \to -\infty} N_h(t + 2k \pi i)$ for $k \in \Z$. 
\end{rem}

The following corollary shows explicit values of parameters for which the assumptions of Theorem~D are satisfied. 

\begin{CorE} If $a,b,c \in \R$, $a,b \neq 0$ and the map $h(z) = ae^z + b z + c$ has a real zero, then its Newton's method $N_h$ has no wandering domains. The condition hold if and only if $ab > 0$ or $ab < 0$, $c/b \ge 1 - \ln(-b/a)$ $($equivalently, if $\beta > 0$ or $\beta < 0$, $\alpha \le \beta(1 - \ln(-\beta)))$. 
\end{CorE}

To prove Theorem~D we will apply Corollary~C. Hence, the first step is to show that $N_h$ is a topologically hyperbolic map. Precisely, we have the following. 

\begin{prop}\label{prop:tophyp} 
If $\alpha, \beta \in \C$, $\beta \neq 0$ and $\inf_{n\ge 0}\dist(N_h^n(u), J(N_h)) > 0$, then $N_h$ is topologically hyperbolic.
\end{prop}
\begin{proof}

As $N_h$ is Newton's method of $h$, the fixed points of $N_h$ coincide with the  zeroes of $h$. We have 
\[
h(z) = e^z + \beta z + \alpha, \qquad h'(z) = e^z + \beta, \qquad h''(z) = e^z, \qquad N_h'(z) = \frac{h(z)h''(z)}{(h'(z))^2}.
\]
In particular, $h'' \neq 0$, so all critical points of $N_h$, denoted by $c_k$, $k \in \Z$, are superattracting fixed points of $N_h$ and satisfy
\begin{equation}\label{eq:c_k}
h(c_k) = e^{c_k} +\beta c_k + \alpha = 0. 
\end{equation}
If $e^{1-\alpha/\beta} + \beta = 0$, then the point 
\begin{equation}\label{eq:ctilde}
\tilde c = 1- \alpha/\beta
\end{equation}
is the unique double zero of $h$ and the unique (attracting and non-superattracting) fixed point of $N_h$ outside $\{c_k\}_{k\in \Z}$. Otherwise, all zeroes of $h$ are simple and the fixed points of $N_h$ coincide with the points $c_k$, $k\in \Z$. 

Therefore, the set of the singular values of $N_h$ consists of the superattracting fixed points $c_k$, $k \in \Z$, and the asymptotic value $u$. By hypothesis, it is enough to show $\inf_{k\in\Z}\dist(c_k, J(f)) > 0$. To this end, we will prove that there exists $r >0$ such that
\begin{equation}\label{eq:N'<}
|N_h'(z)| < 1 \qquad \text{for every } z \in \D(c_k, r), \quad k \in \Z.
\end{equation}
Indeed, in this case $\D(c_k, r)$ is forward invariant, so it is contained in the basin of $c_k$, which ends the proof.

To show \eqref{eq:N'<}, note that by \eqref{eq:c_k}, if $z \in \D(c_k, 1)$, then 
\[
|e^z + \beta z + \alpha| = |e^z - e^{c_k} + \beta (z - c_k)| \le |e^z|(1 + e^{\Re(c_k - z)}) + |\beta||z- c_k| \le (e + 1)|e^z| + |\beta|
\]
and
\[
|e^z| = |e^{z-c_k}| |e^{c_k}| = e^{\Re(z-c_k)}||\beta c_k + \alpha| \ge e^{-1}|\beta c_k + \alpha| \xrightarrow[|k| \to \infty]{} \infty.
\]
This implies that if $z \in \D(c_k, 1)$ for large $|k|$, then
\begin{align*}
|N_h''(z)| &= \left|\frac{h''(z)}{h'(z)} + \frac{h(z)h'''(z)}{(h'(z))^2} - \frac{2h(z)(h''(z))^2}{(h'(z))^3}\right| \\
&\le \left|\frac{e^z}{e^z +\beta}\right| + \left|\frac{e^z(e^z+\beta z + \alpha)}{(e^z + \beta)^2}\right| + \left|\frac{2e^{2z}(e^z + \beta z + \alpha)}{(e^z + \beta)^3}\right| \le C 
\end{align*}
for some $C >1$. Since $N_h'(c_k) = 0$, for $z \in \D(c_k, 1/C)$ we have
\[
|N_h'(z)| \le \sup_{\zeta \in \D(c_k, 1/C)} |N_h''(\zeta)| |z - c_k| < 1,
\]
which gives \eqref{eq:N'<} with $r = 1/C$ and ends the proof.
\end{proof}

\begin{proof}[Proof of Theorem~\rm D] By Proposition~\ref{prop:tophyp}, we know that under our assumptions $N_h$ is topologically hyperbolic. We want to show that $N_h$ has no wandering domains. 

We assume, to get a contradiction, that $W_0$ is a wandering domain. Let $w_0\in W_0$ and let $w_n:=N_h^n(w_0)\subset W_n$ where $W_n$ denotes the Fatou component such that $N_h^n(W) \subset W_n$, $n\geq 0$. Recall that $W_n,\ n\geq 0$ must be simply connected (see \cite{berter}, \cite{bfjknew}). The idea of the proof is to show that for $n$ large enough there exists a vertical segment $\eta$ contained in $W_n$ such that its image by $N_h$ contains a closed curve surrounding $\eta$, which implies that $W_{n+1}$ is not simply connected, a contradiction. 

Now we proceed to the technical details. By Corollary~C, 
\begin{equation}\label{eq:D}
\D(w_n, R_n) \subset W_n
\end{equation}
for some sequence $R_n \to \infty$. Note that the poles of $N_h$ are the solutions of the equation $e^z +\beta = 0$ (with the exception of the point $\tilde c$ defined in \eqref{eq:ctilde} in the case $e^{1-\alpha/\beta} + \beta = 0$), so they are located at the points $\ln|\beta| +i\arg(-\beta) +  2k\pi i$, $k \in \Z$. In particular, since the poles of $N_h$ are outside $W_n$, \eqref{eq:D} implies 
\begin{equation}\label{eq:>}
|\Re(w_n)| \to \infty.
\end{equation}
Note that if $\Re(w_n) < -c$ for a large $c > 0$, then
\[
\frac{|w_{n+1}|}{|w_n|} \le \frac{|w_n|+1}{|w_n|(|\beta|e^{-\Re(w_n)}-1)} + \frac{|\alpha|}{|w_n|(|\beta| - e^{\Re(w_n)})}\le \frac{1+1/c}{|\beta|e^c-1} + \frac{|\alpha|}{c(|\beta|- e^{-c})} < \frac{1}{2},
\]
which together with \eqref{eq:>} implies 
\begin{equation}\label{eq:Re>}
\limsup_{n\to\infty} \Re(w_n) = \infty.
\end{equation}
Moreover, if $\Re(w_n) > c$ for a large $c > 0$ and $|\Im(w_n)| \le e^{\Re(w_n)}/(3|\beta|)$, then, 
since
\[
w_{n+1} = w_n - 1 - \frac{\alpha + \beta(w_n-1)}{e^{w_n}+\beta},
\]
we have
\begin{align*}
|\Re(w_{n+1}) - \Re(w_n) +1| &= \left|\re\left(\frac{\alpha+\beta (w_n + 1)}{e^{w_n} + \beta}\right)\right|\\
&\le \frac{|\alpha| + |\beta|(\Re(w_n) + 1 + |\Im(w_n)|)}{e^{\Re(w_n)} - |\beta|}\\ &\le \frac{|\alpha| + |\beta|(c + 1) + e^c/3}{e^c - |\beta|} < \frac{1}{2},
\end{align*}
so $-3/2 < \Re(w_{n+1}) - \Re(w_n) < -1/2$. This together with \eqref{eq:>} and \eqref{eq:Re>} implies that for a sequence $n_k \to \infty$ there holds
\begin{equation}\label{eq:Im>}
\Re(w_{n_k}) \to \infty, \quad |\Im(w_{n_k})| > \frac{e^{\Re(w_{n_k})}}{3|\beta|}.
\end{equation}
Fix a large $k > 0$ and let 
\[
M = R_{n_k}, \qquad p = \ln (3|\beta| M) , \qquad R = \Re(w_{n_k}) - p. 
\]
Take also $q \in [-2\pi, 2\pi]$ such that $\Im(w_{n_k}) + q = 2 l\pi$ for some $l \in \Z$ and $2|l| \pi \ge |\Im(w_{n_k})|$. Define
\[
z_0 = w_{n_k} - p + iq = R + 2l \pi i.
\]
Then by \eqref{eq:Im>}, 
\begin{equation}\label{eq:z>}
\frac{|z_0|}{e^R} \geq \frac{|\Im(z_0)|}{e^R} = \frac{2|l| \pi i}{e^R} 
\ge \frac{|\Im(w_{n_k})|}{e^R} > \frac{e^{\Re(w_{n_k})}}{3|\beta|e^R} = \frac{e^p}{3|\beta|} = M.
\end{equation}
Note that by \eqref{eq:D} and \eqref{eq:Im>}, we have $p > 0$ and we can assume that the constants $M, R > 0$ are arbitrarily large.
Let 
\[
\eta(t) = z_0 + it, \qquad  t \in [-3\pi, 3\pi]
\]
and note that $\eta \subset \D(w_{n_k}, p + |q| + 3\pi)$, where
\[
p + |q| + 3\pi \le \ln(3|\beta|R_{n_k}) + 5 \pi < R_{n_k}
\]
(if $k$ is chosen large enough), so by \eqref{eq:D},
\begin{equation}\label{eq:eta}
\eta \subset W_{n_k}.
\end{equation}
Let
\[
\gamma(t) = N_h(\eta(t)) - z_0 + 1  - \zeta,
\]
where
\[
\zeta = \frac{z_0}{(e^{2R}/|\beta|^2) - 1}.
\]
By the definition of $N_h$, we can write $\gamma$ in the form
\[
\gamma(t) = \gamma_1(t) + \gamma_2(t) + it
\]
for
\[
\gamma_1(t) = \frac{-z_0}{1 + e^{R + it}/\beta} - \zeta, \qquad  \gamma_2 (t) = \frac{\beta -\alpha - i\beta t}{\beta + e^{R + it}}.
\]
Note that when $t$ varies along an interval of length $2\pi$, the curve $\gamma_1$ describes the image of the circle $\bd\D(0, e^R)$ under the M\"obius map $z \mapsto -z_0/(1 +z/\beta) - \zeta$. Easy computations show that this image is the circle $\bd\D(0, r)$ for
\[
r = \frac{|z_0|}{e^R/|\beta| - |\beta|e^{-R}}.                                                                                                                                                                                                                                                          \]
Hence, 
\begin{equation}\label{eq:g1}
|\gamma_1(t)| = r \ge \frac{|\beta||z_0|}{e^R}
\end{equation}
and
\begin{equation}\label{eq:g2}
|\gamma_2(t)| \le \frac{|\alpha| + (3\pi + 1)|\beta|}{e^R - |\beta|} < \varepsilon_1
\end{equation}
for a small $\varepsilon_1>0$, if we can choose $R$ large enough. By \eqref{eq:z>}, \eqref{eq:g1} and \eqref{eq:g2},
\begin{equation}\label{eq:g>}
|\gamma(t)| \ge |\gamma_1(t)| - |\gamma_2(t)| - 3\pi \ge \frac{|\beta| |z_0|}{e^R} - \varepsilon_1 - 3\pi >  \frac{|\beta| |z_0|}{2e^R},
\end{equation}
if $M$ is chosen large enough.
Let 
\[
\tilde \gamma(t) = \frac{-\beta z_0}{e^{R+it}} 
\]
and note that, by \eqref{eq:g2} and \eqref{eq:z>},
\begin{align*}
\left|\frac{\gamma(t)}{\tilde\gamma(t)} - 1 \right| &\le \left|\frac{\gamma_1(t) + \zeta}{\tilde\gamma(t)} - 1 \right| + \left|\frac{\zeta}{\tilde\gamma(t)}\right|+ \left|\frac{\gamma_2(t)}{\tilde\gamma(t)}\right| + \frac{3\pi}{|\tilde\gamma(t)|}\\
&\le \frac{1}{|1 + e^{R+it}/\beta|} + \frac{|\beta| e^R}{e^{2R} - |\beta|^2} + \frac{(|\alpha| + (3\pi + 1)|\beta|)e^R}{|\beta| |z_0| (e^R - |\beta|)} + \frac{3\pi e^R}{|\beta||z_0|}\\
&\le \frac{1}{e^R/|\beta| - 1} + \frac{|\beta|}{e^R - |\beta|^2e^{-R}} + \frac{|\alpha/\beta| + 3\pi +1}{M(e^R - |\beta|)} + \frac{3\pi}{|\beta|M} < \varepsilon_2
\end{align*}
for a small $\varepsilon_2 >0$, if $M, R$ are chosen sufficiently large. Hence,
there exists a branch $\Arg$ of the argument defined along the curve $\gamma(t)$, such that
\begin{equation}\label{eq:Arg}
\Arg(\gamma(t))= -t + A + \delta(t),
\end{equation}
where 
\[
A = \Arg(-\beta z_0) \in [0, 2\pi]
\]
and
\[
|\delta(t)| < \varepsilon_2.
\]
By \eqref{eq:g1} and \eqref{eq:g2}, 
\begin{equation}\label{eq:mod}
\gamma_1 + \gamma_2 \subset \D(0, r + \varepsilon_1) \setminus \D(0, r - \varepsilon_1)
\end{equation}
for a small $\varepsilon_1 > 0$ (provided $R$ is chosen large enough) and by \eqref{eq:Arg}, there exist $t_1^+, t_2^+, t_1^-, t_2^- \in [-3\pi, 3\pi]$ such that 
\begin{align*}
t_1^+ &= \min\{t \in [-\pi/2 + A - \varepsilon_2, -\pi/2 + A + \varepsilon_2]: \Arg(\gamma(t)) = \pi/2\},\\
t_2^+ &= \min\{t\in [-5\pi/2 + A - \varepsilon_2, -5\pi/2 + A + \varepsilon_2]: \Arg(\gamma(t)) = 5\pi/2\},\\
t_1^- &= \max\{t \in [\pi/2 + A - \varepsilon_2, \pi/2 + A + \varepsilon_2]: \Arg(\gamma(t)) = -\pi/2\},\\
t_2^- &= \max\{t \in [-3\pi/2 + A - \varepsilon_2, -3\pi/2 + A + \varepsilon_2]: \Arg(\gamma(t)) = 3\pi/2\}.
\end{align*}

\begin{figure}[t]
    \centering
     \includegraphics[width=0.45\textwidth]{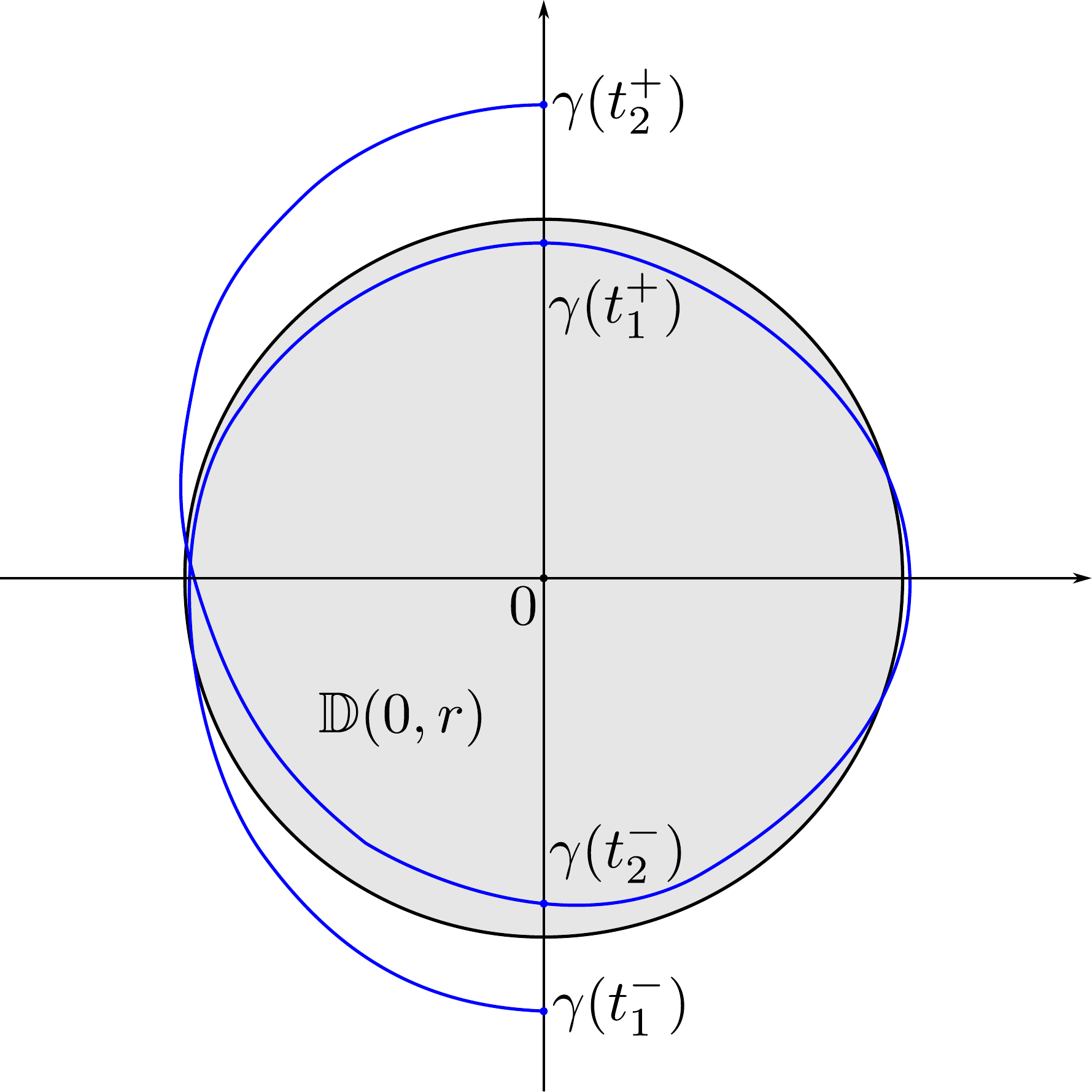}
     \label{fig:gamma}\caption{The curve $\gamma$.}
    \end{figure}
By definition, 
\begin{align*}
\gamma(t_1^+), \gamma(t_2^+) &\in \{z: \Re(z) = 0, \, \Im(z) > 0\},\\
\gamma(t_1^-), \gamma(t_2^-) &\in \{z: \Re(z) = 0, \, \Im(z) < 0\},\\
\gamma|_{(t_j^-, t_j^+)} &\subset \{z: \Re(z) < 0\} \quad \text{for } j = 1, 2.
\end{align*}
Since $\gamma(t) = \gamma_1(t) + \gamma_2(t) + it$, by \eqref{eq:Arg} and \eqref{eq:mod} we have 
\[
|\gamma(t_1^+)| - |\gamma(t_2^+)| = \Im(\gamma(t_1^+)) - \Im(\gamma(t_2^+)) < t_1^+ - t_2^+ + 2 \varepsilon_1 < -2\pi + 2 \varepsilon_2 + 2 \varepsilon_1 < 0
\]
and
\[
|\gamma(t_1^-)| - |\gamma(t_2^-)| = -\Im(\gamma(t_2^-)) + \Im(\gamma(t_2^-)) > t_2^- - t_1^- - 2 \varepsilon_1 > 2\pi - 2 \varepsilon_2 - 2 \varepsilon_1 > 0
\]
(see Figure~\ref{fig:gamma}).
Hence, by standard topological arguments, the curves $\gamma|_{(t_1^-, t_1^+)}$ and $\gamma|_{(t_2^-, t_2^+)}$ intersect. Together with \eqref{eq:g>}, this shows that $\D\left(0, |\beta| |z_0|/(2e^R)\right)$ lies in a bounded component of $\C \setminus \gamma$, so $D:=\D\left(z_0-1+\zeta, |\beta| |z_0|/(2e^R)\right)$ lies in a bounded component of $\C \setminus N_h(\eta)$. By \eqref{eq:z>}, we have
\[
|\zeta - 1| \le |\zeta| + 1 =  \frac{|z_0|}{(e^{2R}/|\beta|^2) - 1} + 1 
\le \frac{|\beta||z_0|}{4e^R} + \frac{|\beta|M}{4} \le \frac{|\beta||z_0|}{2e^R}
\]
(provided $M, R$ are chosen sufficiently large), which implies $z_0 \in D$, so $z_0$ in a bounded component of $\C \setminus N_h(\eta)$.
By \eqref{eq:eta}, we have $z_0 \in W_{n_k} \subset \C \setminus W_{n_k+1}$ and $f(\eta) \subset W_{n_k+1}$. Hence, $W_{n_k+1}$ is not simply connected, which makes a contradiction.  
\end{proof}

This ends the proof of Theorem~D. We finish this section proving Corollary~E.

\begin{proof}[Proof of Corollary~\rm E]
We will check that under the hypothesis, the unique asymptotic value $u$ of $N_h$ is in the basin of attraction of a real fixed point of $N_h$ and so 
$$
\inf_{n\ge 0}\dist\left(N_h^n(u), J(N_h)\right) > 0.
$$
Consider $N_h$ restricted to the real axis. Note first that since
\[
N_h(t) = t - 1 - \frac{\alpha + \beta (t-1)}{e^t + \beta},
\]
we have $N_h(t) < t$ for sufficiently large $t  \in \R$.

Assume first $\beta > 0$. Then $h$ has one (simple) real zero and $h'$ has no real zeroes, so $N_h|_{\R}$ has no poles and exactly one fixed point, say $c_0$, which is superattracting, since $e^{1-\alpha/\beta} + \beta > 0$ (see the proof of Proposition~\ref{prop:tophyp}) and is also the unique critical point of $N_h|_{\R}$, such that $N_h''(c_0) \neq 0$. Then $N_h(t) < t$ for $t > c_0$ and $N_h|_{\R}$ attains the minimum at $c_0$. As $u = \lim_{t \to -\infty} N_h(t)$, we have $u > c_0$, so $N_h^n(u)$ converges monotonically to $c_0$ as $n \to \infty$.

Assume from now on $\beta < 0$. Consider first the case $\alpha = \beta(1 - \ln(-\beta))$. Then $e^{1-\alpha / \beta} + \beta = 0$, so $h$ has one real zero equal to $\tilde c$ defined in \eqref{eq:ctilde}, which is also a zero of $h'$. Hence, $N_h|_{\R}$ has no poles, no critical points and exactly one fixed point $\tilde c$, which is attracting. Consequently, $N_h|_{\R}$ is increasing and $N_h^n(u)$ converges monotonically to $\tilde c$ as $n \to \infty$.

Consider now the case $\alpha < \beta(1 - \ln(-\beta))$. Then $h$ has two (simple) real zeroes, so $N_h|_{\R}$ has one (simple) pole, say $p$ and exactly two superattracting fixed points, say $c_0 < c_1$, which are the unique critical points of $N_h|_{\R}$, such that $N_h''(c_0), N_h''(c_1) \neq 0$. Since $N_h(t) < t$ for large $t$ and $u = \lim_{t \to -\infty} N_h(t)$, we have $c_0 < p < c_1$, $N_h(c_0) = \max_{t < p} N_h(t)$ and $N_h(c_1) = \min_{t > p} N_h(t)$. Hence, $u < N_h(c_0)$ and $N_h^n(u)$ converges monotonically to $c_0$ as $n \to \infty$.

Finally, it is straightforward to check that in the case $\alpha < \beta(1 - \ln(-\beta))$ the map $h$ has no real zeroes.

\end{proof}

\bibliography{wandering}
\end{document}

%% file: baker.pdf_tex
%% Creator: Inkscape inkscape 0.91, www.inkscape.org
%% Accompanies image file 'baker.pdf' (pdf, eps, ps)
\begingroup%
  \makeatletter%
  \providecommand\color[2][]{%
    \errmessage{(Inkscape) Color is used for the text in Inkscape, but the package 'color.sty' is not loaded}%
    \renewcommand\color[2][]{}%
  }%
  \providecommand\transparent[1]{%
    \errmessage{(Inkscape) Transparency is used (non-zero) for the text in Inkscape, but the package 'transparent.sty' is not loaded}%
    \renewcommand\transparent[1]{}%
  }%
  \providecommand\rotatebox[2]{#2}%
  \ifx\svgwidth\undefined%
    \setlength{\unitlength}{372.77122267bp}%
    \ifx\svgscale\undefined%
      \relax%
    \else%
      \setlength{\unitlength}{\unitlength * \real{\svgscale}}%
    \fi%
  \else%
    \setlength{\unitlength}{\svgwidth}%
  \fi%
  \global\let\svgwidth\undefined%
  \global\let\svgscale\undefined%
  \makeatother%
  \begin{picture}(1,0.7267737)%
    \put(0,0){\includegraphics[width=\unitlength]{baker.pdf}}%
    \put(0.16755138,0.30191){\color[rgb]{0,0,0}\makebox(0,0)[lb]{\smash{$U$}}}%
    \put(0.4545907,0.38399791){\color[rgb]{0,0,0}\makebox(0,0)[lb]{\smash{$\Gamma$}}}%
    \put(0.22066707,0.10661595){\color[rgb]{0,0,0}\makebox(0,0)[lb]{\smash{$r_1(k)$}}}%
    \put(0.72687567,0.00253067){\color[rgb]{0,0,0}\makebox(0,0)[lb]{\smash{$r_{l_k}(k) = r_1(k+1)$}}}%
    \put(0.90151361,0.30995783){\color[rgb]{0,0,0}\makebox(0,0)[lb]{\smash{$f^{pN(k+1)}(z)$}}}%
    \put(0.19685365,0.39441405){\color[rgb]{0,0,0}\makebox(0,0)[lb]{\smash{$f^{pN(k)}(z)$}}}%
    \put(0.4890892,0.59885199){\color[rgb]{0,0,0}\makebox(0,0)[lb]{\smash{$w_j(k)$}}}%
    \put(0.52916727,0.46876841){\color[rgb]{0,0,0}\makebox(0,0)[lb]{\smash{$\gamma_j(k)$}}}%
    \put(0.53162767,0.30189605){\color[rgb]{0,0,0}\makebox(0,0)[lb]{\smash{$\zeta_j(k)$}}}%
    \put(0.39007788,0.59524629){\color[rgb]{0,0,0}\makebox(0,0)[lb]{\smash{$u_j(k)$}}}%
  \end{picture}%
\endgroup%